\documentclass[amstex,12pt, amssymb]{article}

\usepackage{mathtext}
\usepackage[cp1251]{inputenc}
\usepackage[T2A]{fontenc}
\usepackage[dvips]{graphicx}
\usepackage{amsmath}
\usepackage{amssymb}
\usepackage{amsxtra}
\usepackage{latexsym}
\usepackage{ifthen}

\newcounter{lemma}[section]

\newcounter{corollary}[section]

\newcounter{remark}[section]

\newcounter{theorem}[section]

\newcounter{proposition}[section]

\newcounter{example}

\numberwithin{equation}{section}

\begin{document}

\markboth{E~.SEVOST'YANOV, V.~TARGONSKII}{\centerline{AN ANALOGUE OF
KOEBE'S THEOREM ...}}

\def\cc{\setcounter{equation}{0}
\setcounter{figure}{0}\setcounter{table}{0}}

\overfullrule=0pt


\author{EVGENY SEVOST'YANOV, VALERY TARGONSKII}

\title{
{\bf AN ANALOGUE OF KOEBE'S THEOREM AND THE OPENNESS OF A LIMIT MAP
IN ONE CLASS}}

\date{\today}
\maketitle

\begin{abstract}
We study mappings that satisfy the inverse modulus inequality of
Poletsky type in a fixed domain. It is shown that, under some
additional restrictions, the image of a ball under such mappings
contains a fixed ball uniformly over the class. This statement can
be interpreted as the well-known analogue of Koebe's theorem for
analytic functions. As an application of the obtained result, we
show that, if a sequence of mappings belonging to the specified
class converges locally uniformly, then the limit mapping is open.
\end{abstract}

\bigskip
{\bf 2010 Mathematics Subject Classification: Primary 30C65;
Secondary 31A15, 31B25}

\medskip
{\bf Key words: mappings  with a finite and bounded distortion,
moduli, capacity}

\section{Introduction}

Let us recall the formulation of the classical Koebe theorem, see,
for example,~\cite[Theorem~1.3]{CG}.

\medskip
{\bf Theorem A.} {\it Let $f:{\Bbb D}\rightarrow {\Bbb C}$ be an
univalent analytic function such that $f(0)=0$ and
$f^{\,\prime}(0)=1.$ Then the image of $f$ covers the open disk
centered at $0$ of radius one-quarter, that is, $f({\Bbb D})\supset
B(0, 1/4).$}

\medskip
The main fact contained in the paper is the statement that something
similar has been done for a much more general class of spatial
mappings. Below $dm(x)$ denotes the element of the Lebesgue measure
in ${\Bbb R}^n.$ Everywhere further the boundary $\partial A $ of
the set $A$ and the closure $\overline{A}$ should be understood in
the sense of the extended Euclidean space $\overline{{\Bbb R}^n}.$
Recall that, a Borel function $\rho:{\Bbb R}^n\,\rightarrow
[0,\infty] $ is called {\it admissible} for the family $\Gamma$ of
paths $\gamma$ in ${\Bbb R}^n,$ if the relation
\begin{equation}\label{eq1.4}
\int\limits_{\gamma}\rho (x)\, |dx|\geqslant 1
\end{equation}
holds for all (locally rectifiable) paths $ \gamma \in \Gamma.$ In
this case, we write: $\rho \in {\rm adm} \,\Gamma .$  The {\it
modulus} of $\Gamma $ is defined by the equality
\begin{equation}\label{eq1.3gl0}
M(\Gamma)=\inf\limits_{\rho \in \,{\rm adm}\,\Gamma}
\int\limits_{{\Bbb R}^n} \rho^n (x)\,dm(x)\,.
\end{equation}
Let $y_0\in {\Bbb R}^n,$ $0<r_1<r_2<\infty$ and
\begin{equation}\label{eq1**}
A=A(y_0, r_1,r_2)=\left\{ y\,\in\,{\Bbb R}^n:
r_1<|y-y_0|<r_2\right\}\,.\end{equation}
Given $x_0\in{\Bbb R}^n,$ we put
$$B(x_0, r)=\{x\in {\Bbb R}^n: |x-x_0|<r\}\,, \quad {\Bbb B}^n=B(0, 1)\,,$$
$$S(x_0,r) = \{
x\,\in\,{\Bbb R}^n : |x-x_0|=r\}\,. $$
A mapping $f: D \rightarrow{\Bbb R}^n$ is called {\it discrete} if
the pre-image $\{f^{-1}\left(y\right)\}$ of any point $y\,\in\,{\Bbb
R}^n$ consists of isolated points, and {\it open} if the image of
any open set $U\subset D$ is an open set in ${\Bbb R}^n.$

Given sets $E,$ $F\subset\overline{{\Bbb R}^n}$ and a domain
$D\subset {\Bbb R}^n$ we denote by $\Gamma(E,F,D)$ the family of all
paths $\gamma:[a,b]\rightarrow \overline{{\Bbb R}^n}$ such that
$\gamma(a)\in E,\gamma(b)\in\,F$ and $\gamma(t)\in D$ for $t \in (a,
b).$ Given a mapping $f:D\rightarrow {\Bbb R}^n,$ a point $y_0\in
\overline{f(D)}\setminus\{\infty\},$ and
$0<r_1<r_2<r_0=\sup\limits_{y\in f(D)}|y-y_0|,$ we denote by
$\Gamma_f(y_0, r_1, r_2)$ a family of all paths $\gamma$ in $D$ such
that $f(\gamma)\in \Gamma(S(y_0, r_1), S(y_0, r_2),
A(y_0,r_1,r_2)).$ Let $Q:{\Bbb R}^n\rightarrow [0, \infty]$ be a
Lebesgue measurable function. We say that {\it $f$ satisfies the
inverse Poletsky inequality at a point $y_0\in
\overline{f(D)}\setminus\{\infty\}$} if the relation
\begin{equation}\label{eq2*A}
M(\Gamma_f(y_0, r_1, r_2))\leqslant \int\limits_{A(y_0,r_1,r_2)\cap
f(D)} Q(y)\cdot \eta^{n}(|y-y_0|)\, dm(y)
\end{equation}
holds for any Lebesgue measurable function $\eta:
(r_1,r_2)\rightarrow [0,\infty ]$ such that
\begin{equation}\label{eqA2}
\int\limits_{r_1}^{r_2}\eta(r)\, dr\geqslant 1\,.
\end{equation}
The definition of the relation~(\ref{eq2*A}) at the point
$y_0=\infty$ may be given by the using of the inversion
$\psi(y)=\frac{y}{|y|^2}$ at the origin.

\medskip
Note that conformal mappings preserve the modulus of families of
paths, so that we may write
$$M(\Gamma)=M(f(\Gamma))\,.$$
It is not difficult to see from this that conformal mappings from
Koebe theorem satisfy the relation~(\ref{eq2*A}) with $Q\equiv 1$
for any function $\eta$ in~(\ref{eqA2}).

\begin{remark}\label{rem1}
It is known that the quasiregular mappings satisfy the inequality
$$M(\Gamma)\leqslant N(f, D)K_O(f)M(f(\Gamma))\,,$$
where $1\leqslant K_O(f)<\infty$ is some number, and $N(f, D)$
denotes the multiplicity function,
$$
N(y,f,E)\,=\,{\rm card}\,\left\{x\in E: f(x)=y\right\}\,,
$$
\begin{equation}\label{eq1G}
N(f,E)\,=\,\sup\limits_{y\in{\Bbb R}^n}\,N(y,f,E)\,,
\end{equation}
see~\cite[Theorem~3.2]{MRV$_1$}. There are also mappings in which
the distortion of the modulus of families of paths is much more
complex. Say, for homeomorphisms $f\in W^{1, n}_{\rm loc}$ such that
$f^{\,-1}\in W^{1, n}_{\rm loc}$ we have the inequality
\begin{equation}\label{eq16A}
M(\Gamma)\leqslant \int\limits_{f(D)}K_I(y, f^{\,-1})\cdot
\rho_*^n(y)\,dm(x)
\end{equation}
for any $\rho_* \in {\rm adm}\, f(\Gamma) $ (see below), where
\begin{equation}\label{eq15A}
K_{I}(y, f^{\,-1})\quad =\sum\limits_{x\in f^{\,-1}(y)}K_O(x, f)\,,
\end{equation}
$$K_{O}(x,f)\quad =\quad \left\{
\begin{array}{rr}
\frac{\Vert f^{\,\prime}(x)\Vert^n}{|J(x,f)|}, & J(x,f)\ne 0,\\
1,  &  f^{\,\prime}(x)=0, \\
\infty, & \text{otherwise}
\end{array}
\right.\,\,,$$
see \cite[Theorems~8.1, 8.6]{MRSY}.

All of the above allows us to assert that relation~(\ref{eq2*A}) is
satisfied by a fairly large number of mappings. In general, for
practically all currently known classes, including conformal and
quasiconformal mappings, quasiregular mappings, mappings with finite
distortion, etc. such inequalities are satisfied.
\end{remark}

\medskip
Set
\begin{equation}\label{eq12}
q_{y_0}(r)=\frac{1}{\omega_{n-1}r^{n-1}}\int\limits_{S(y_0,
r)}Q(y)\,d\mathcal{H}^{n-1}(y)\,, \end{equation}
and $\omega_{n-1}$ denotes the area of the unit sphere ${\Bbb
S}^{n-1}$ in ${\Bbb R}^n.$

\medskip
We say that a function ${\varphi}:D\rightarrow{\Bbb R}$ has a {\it
finite mean oscillation} at a point $x_0\in D,$ write $\varphi\in
FMO(x_0),$ if
$$\limsup\limits_{\varepsilon\rightarrow
0}\frac{1}{\Omega_n\varepsilon^n}\int\limits_{B( x_0,\,\varepsilon)}
|{\varphi}(x)-\overline{{\varphi}}_{\varepsilon}|\ dm(x)<\infty\,,
$$
where $\overline{{\varphi}}_{\varepsilon}=\frac{1}
{\Omega_n\varepsilon^n}\int\limits_{B(x_0,\,\varepsilon)}
{\varphi}(x) \,dm(x)$ and $\Omega_n$ is the volume of the unit ball
${\Bbb B}^n$ in ${\Bbb R}^n.$
We also say that a function ${\varphi}:D\rightarrow{\Bbb R}$ has a
finite mean oscillation at $A\subset \overline{D},$ write
${\varphi}\in FMO(A),$ if ${\varphi}$ has a finite mean oscillation
at any point $x_0\in A.$ Let $h$ be a chordal metric in
$\overline{{\Bbb R}^n},$
$$h(x,\infty)=\frac{1}{\sqrt{1+{|x|}^2}}\,,$$
\begin{equation}\label{eq3C}
h(x,y)=\frac{|x-y|}{\sqrt{1+{|x|}^2} \sqrt{1+{|y|}^2}}\qquad x\ne
\infty\ne y\,.
\end{equation}
and let $h(E):=\sup\limits_{x,y\in E}\,h(x,y)$ be a chordal diameter
of a set~$E\subset \overline{{\Bbb R}^n}$ (see, e.g.,
\cite[Definition~12.1]{Va}).

\medskip
Given a continuum $E\subset D,$ $\delta>0$ and a Lebesgue measurable
function $Q:{\Bbb R}^n\rightarrow [0, \infty]$ we denote by
$\frak{F}_{E, \delta}(D)$ the family of all mapping $f:D\rightarrow
{\Bbb R}^n,$ $n\geqslant 2,$ satisfying
relations~(\ref{eq2*A})--(\ref{eqA2}) at any point $y_0\in
\overline{{\Bbb R}^n}$ such that $h(f(E))\geqslant \delta.$ The
following statement holds.

\medskip
\begin{theorem}\label{th1}
{\it Let $D$ be a domain in ${\Bbb R}^n,$ $n\geqslant 2,$ and let
$B(x_0, \varepsilon_1)\subset D$ for some $\varepsilon_1>0.$

\medskip
Assume that, $Q\in L^1({\Bbb R}^n)$ and, in addition, one of the
following conditions hold:

\medskip
1) $Q\in FMO(\overline{{\Bbb R}^n});$

\medskip
2) for any $y_0\in \overline{{\Bbb R}^n}$ there is $\delta(y_0)>0$
such that
\begin{equation}\label{eq5D}
\int\limits_{0}^{\delta(y_0)}
\frac{dt}{tq_{y_0}^{\frac{1}{n-1}}(t)}=\infty\,.
\end{equation}
Then there is $r_0>0,$ which does not depend on $f,$ such that
$$f(B(x_0, \varepsilon_1))\supset B(f(x_0), r_0)\qquad \forall\,\,f\in \frak{F}_{E,
\delta}(D)\,.$$}
\end{theorem}
\begin{remark}
The condition $Q\in FMO(\infty)$ of the condition~(\ref{eq5D}) for
$y_0=\infty$ must be understood as follows: these conditions hold
for $y_0=\infty$ if and only if the function
$\widetilde{Q}:=Q\left(\frac{y}{|y|^2}\right)$ satisfies similar
conditions at the origin.
\end{remark}

\medskip
Note that the above analogue of Koebe's theorem has an important
application in the field of convergence of mappings. Recall that, a
mapping $f:D\rightarrow {\Bbb R}^n$ is called a {\it
$K$-quasiregular mapping,} if the following conditions hold:

\medskip
\label{page1} 1) $f\in W_{loc}^{1,n}(D),$

2) the Jacobian $J(x,f)$ of $f$ at $x\in D$ preserves the sign
almost everywhere in $D,$

3) $\Vert f^{\,\prime}(x) \Vert^n \leqslant K \cdot |J(x,f)|$ for
almost any $x\in D$ and some constant $K<\infty,$ where
$\Vert f^{\,\prime}(x)\Vert\,=\,\max\limits_{h\in {\Bbb R}^n
\backslash \{0\}} \frac {|f^{\,\prime}(x)h|}{|h|}\,,\quad
J(x,f)=\det f^{\,\prime}(x),$
see e.g. \cite[Section~4, Ch.~I]{Re}, cf.~\cite[Definition~2.1,
Ch.~I]{Ri}. As is known, the class of mappings with bounded
distortion is closed under locally uniform convergence. In
particular, the following statement is true (see, for example,
\cite[Theorem~9.2.II]{Re}).

\medskip
{\bf Theorem B.} Let $f_j:D\rightarrow {\Bbb R}^n,$ $n\geqslant 2,$
$j=1,2,\ldots,$ be a sequence of $K$-quasiregular mappings
converging to some mapping $f:D\rightarrow {\Bbb R}^n$ as
$j\rightarrow\infty$ locally uniformly in $D.$ Then either $f$ is
$K$-quasiregular, of $f$ is a constant. In particular, in the first
case $f$ is discrete and open (see \cite[Theorems~6.3.II
and~6.4.II]{Re}).

\medskip
As for the classes we are studying in~(\ref{eq2*A})--(\ref{eqA2}),
the following analogue of Theorem B is valid for them.

\medskip
\begin{theorem}\label{th2}
{\it\, Let $D$ be a domain in ${\Bbb R}^n,$ $n\geqslant 2.$ Let
$f_j:D\rightarrow {\Bbb R}^n,$ $n\geqslant 2,$ $j=1,2,\ldots,$ be a
sequence of open discrete mappings satisfying the
conditions~(\ref{eq2*A})--(\ref{eqA2}) at any point $y_0\in
\overline{{\Bbb R}^n}$ and converging to some mapping
$f:D\rightarrow {\Bbb R}^n$ as $j\rightarrow\infty$ locally
uniformly in $D.$ Assume that the conditions on the function $Q$
from Theorem~\ref{th1} hold. Then either $f$ is a constant, or $f$
is light and open.}
\end{theorem}

\medskip
\begin{remark}\label{rem2}
The lightness of the mapping $f$ in Theorem~\ref{th2} was
established earlier, see~\cite{Sev$_1$}, cf.~\cite{Cr}. The goal of
the paper is to obtain the openness of this mapping, which will
follow from Theorem~\ref{th1}. Note that mappings that satisfy
conditions~(\ref{eq2*A})--(\ref{eqA2}) may not be open. For example,
let $x=(x_1,\ldots, x_n).$ We define $f$ as the identical mapping in
the closed domain $\{x_n\geqslant 0\}$ and set
$f(x)=(x_1,\ldots,-x_n)$ for $x_n<0.$ Observe that, the mapping $f$
satisfies conditions~(\ref{eq2*A})--(\ref{eqA2}) for $Q(y)\equiv 2.$
Indeed, $f$ preserves the lengths of paths, is differentiable almost
everywhere and has Luzin’s $N$ and $N^{\,-1}$-properties. Therefore,
$f$ is a mapping with a finite length distortion (for the definition
see~\cite[section~8]{MRSY}). Now, $f$ satisfies~(\ref{eq16A}) with
$Q:=K_{I}(y, f^{\,-1})\quad =\sum\limits_{x\in f^{\,-1}(y)}K_O(x,
f)\leqslant 1+1=2$ by \cite[Theorems~8.1, 8.6]{MRSY}. Therefore $f$
satisfies conditions~(\ref{eq2*A})--(\ref{eqA2}) for $Q(y)\equiv 2,$
as well.

As for the discreteness of the limit mapping $f$ in
Theorem~\ref{th2}, whether this mapping will be such is currently
unknown.
\end{remark}

\section{Preliminaries}

The following statement was proved in~\cite[Lemma~2.1]{Na}.

\medskip
\begin{proposition}\label{pr2}
{\it\, The number $\delta_n(r)=\inf M(\Gamma(F, F_*, \overline{{\Bbb
R}^n})),$ where the infimum is taken over all continua $F$ and $F_*$
in $\overline{{\Bbb R}^n}$ with $h(F)\geqslant r$ and
$h(F_*)\geqslant r,$ is positive for each $r>0$ and zero for $r=0.$}
\end{proposition}

\medskip The following statement also may be found in~\cite[Theorem~4.1]{Na}).

\medskip
\begin{proposition}\label{pr5}
{\it\, Let $\frak{F}$ be a collection of connected sets in a domain
$D$ and let $\inf h(F)> 0,$ $F\in \frak{F}.$ Then $\inf\limits_{F\in
\frak{F}}M(\Gamma(F, A, D))>0$ either for each or for no continuum
$A$ in~$D.$}
\end{proposition}

\medskip
The following statement may be found in~\cite[Lemma~4.3]{Vu}.

\medskip
\begin{proposition}\label{pr4}
{\it\, Let $D$ be an open half space or an open ball in ${\Bbb R}^n$
and let $E$ and $F$ be subsets of $D.$ Then $$M(\Gamma(E, F,
D))\geqslant \frac{1}{2}\cdot M(\Gamma(E, F, \overline{{\Bbb
R}^n}))\,.$$}
\end{proposition}
For a domain $D\subset {\Bbb R}^n,$ $n\geqslant 2,$ and a Lebesgue
measurable function $Q:{\Bbb R}^n\rightarrow [0, \infty],$
$Q(y)\equiv 0$ for $y\in{\Bbb R}^n\setminus f(D),$ we denote by
$\frak{F}_Q(D)$ the family of all open discrete mappings
$f:D\rightarrow {\Bbb R}^n$ such that
relations~(\ref{eq2*A})--(\ref{eqA2}) hold for each point $y_0\in
f(D).$ The following result holds (see~\cite[Theorem~1.1]{SSD}).

\medskip
\begin{proposition}\label{pr1}
{\it Let $n\geqslant 2,$ and let $Q\in L^1({\Bbb R}^n).$ Then for
any $x_0\in D$ and any $r_0>0$ such that $0<r_0<{\rm dist}(x_0,
\partial D)$ the inequality
\begin{equation}\label{eq2CB}
|f(x)-f(x_0)|\leqslant\frac{C_n\cdot (\Vert
Q\Vert_1)^{1/n}}{\log^{1/n}\left(1+\frac{r_0}{2|x-x_0|}\right)}
\end{equation}
holds for any $x, y\in B(x_0, r_0)$ and $f\in \frak{F}_Q(D),$ where
$\Vert Q\Vert_1$ denotes the $L^1$-norm of $Q$ in ${\Bbb R}^n,$ and
$C_n>0$ is some constant depending only on $n.$ In particular,
$\frak{F}_Q(D)$ is equicontinuous in $D.$}
\end{proposition}

\medskip
Let $D\subset {\Bbb R}^n,$ $f:D\rightarrow {\Bbb R}^n$ be a discrete
open mapping, $\beta: [a,\,b)\rightarrow {\Bbb R}^n$ be a path, and
$x\in\,f^{\,-1}(\beta(a)).$ A path $\alpha: [a,\,c)\rightarrow D$ is
called a {\it maximal $f$-lifting} of $\beta$ starting at $x,$ if
$(1)\quad \alpha(a)=x\,;$ $(2)\quad f\circ\alpha=\beta|_{[a,\,c)};$
$(3)$\quad for $c<c^{\prime}\leqslant b,$ there is no a path
$\alpha^{\prime}: [a,\,c^{\prime})\rightarrow D$ such that
$\alpha=\alpha^{\prime}|_{[a,\,c)}$ and $f\circ
\alpha^{\,\prime}=\beta|_{[a,\,c^{\prime})}.$ If $\beta:[a,
b)\rightarrow\overline{{\Bbb R}^n}$ is a path and if
$C\subset\overline{{\Bbb R}^n},$ we say that $\beta\rightarrow C$ as
$t\rightarrow b,$ if the spherical distance $h(\beta(t),
C)\rightarrow 0$ as $t\rightarrow b$ (see
\cite[section~3.11]{MRV$_2$}), where $h(\beta(t),
C)=\inf\limits_{x\in C}h(\beta(t), x).$ The following assertion
holds (see~\cite[Lemma~3.12]{MRV$_2$}).

\medskip
\begin{proposition}\label{pr3}
{\it Let $f:D\rightarrow {\Bbb R}^n,$ $n\geqslant 2,$ be an open
discrete mapping, let $x_0\in D,$ and let $\beta: [a,\,b)\rightarrow
{\Bbb R}^n$ be a path such that $\beta(a)=f(x_0)$ and such that
either $\lim\limits_{t\rightarrow b}\beta(t)$ exists, or
$\beta(t)\rightarrow \partial f(D)$ as $t\rightarrow b.$ Then
$\beta$ has a maximal $f$-lifting $\alpha: [a,\,c)\rightarrow D$
starting at $x_0.$ If $\alpha(t)\rightarrow x_1\in D$ as
$t\rightarrow c,$ then $c=b$ and $f(x_1)=\lim\limits_{t\rightarrow
b}\beta(t).$ Otherwise $\alpha(t)\rightarrow \partial D$ as
$t\rightarrow c.$}
\end{proposition}

\medskip
The following statement may be found in~\cite[Lemma~1.3]{Sev$_2$}.

\medskip
\begin{proposition}\label{pr6}
{\it\, Let $Q:{\Bbb R}^n\rightarrow [0,\infty],$ $n\geqslant 2,$ be
a Lebesgue measurable function and let $x_0\in {\Bbb R}^n.$ Assume
that either of the following conditions holds

\noindent (a) $Q\in FMO(x_0),$

\noindent (b)
$q_{x_0}(r)\,=\,O\left(\left[\log{\frac1r}\right]^{n-1}\right)$ as
$r\rightarrow 0,$

\noindent (c) for some small $\delta_0=\delta_0(x_0)>0$ we have the
relations
\begin{equation}\label{eq5***B}
\int\limits_{\delta}^{\delta_0}\frac{dt}{tq_{x_0}^{\frac{1}{n-1}}(t)}<\infty,\qquad
0<\delta<\delta_0,
\end{equation}
and
\begin{equation}\label{eq5**}
\int\limits_{0}^{\delta_0}\frac{dt}{tq_{x_0}^{\frac{1}{n-1}}(t)}=\infty\,.
\end{equation}
Then  there exist a number $\varepsilon_0\in(0,1)$ and a function
$\psi(t)\geqslant 0$ such that the relation
\begin{equation}\label{eqlem21}
\int\limits_{\varepsilon<|x-b|<\varepsilon_0}Q(x)\cdot\psi^n(|x-b|)
\ dm(x)=o(I^n(\varepsilon, \varepsilon_0))\,,
\end{equation}
holds as $\varepsilon\rightarrow 0,$
where $\psi:(0, \varepsilon_0)\rightarrow [0, \infty)$ is some
function such that, for some $0<\varepsilon_1<\varepsilon_0,$
\begin{equation}\label{eqlem22}
0<I(\varepsilon, \varepsilon_0)
=\int\limits_{\varepsilon}^{\varepsilon_0}\psi(t)\,dt < \infty
\qquad \forall\quad\varepsilon \in(0, \varepsilon_1)\,.
\end{equation}}
\end{proposition}

\section{Main Lemmas}

\begin{lemma}\label{lem2}
{\it\, Let $D$ be a domain in ${\Bbb R}^n,$ let $x_0\in D,$ let $A$
be a (non-degenerate) continuum in $D,$ and let $\varepsilon_1>0$ be
such that $B(x_0, \varepsilon_1)\subset D.$ Let $r>0$ and let $C_j,$
$j=1,2,\ldots, $ be a sequence of continua in $B(x_0,
\varepsilon_1)$ such that $h(C_j)\geqslant r,$
$h(C_j)=\sup\limits_{x, y\in C_j}h(x, y).$ Then there is $R_0>0$
such that $$M(\Gamma(C_j, A, D))\geqslant R_0\qquad\forall\,\,j\in
{\Bbb N}\,.$$}
\end{lemma}
\begin{proof}
Let $A_1$ be an arbitrary (non-degenerate) continuum in $B(x_0,
\varepsilon_1).$ By Proposition~\ref{pr2}, there is $R_*>0$ such
that $M(\Gamma(C_j, A_1, \overline{{\Bbb R}^n}))\geqslant R_*$ for
any $j\in {\Bbb N}.$ Now, by Proposition~\ref{pr4} $M(\Gamma(C_j,
A_1, B(x_0, \varepsilon_1)))\geqslant \frac{1}{2}\cdot M(\Gamma(C_j,
A_1, \overline{{\Bbb R}^n}))\geqslant R_*/2.$ Now $M(\Gamma(C_j,
A_1, D))\geqslant R_*/2$ for any $j\in {\Bbb N}.$ Finally,
$M(\Gamma(C_j, A, D))\geqslant R_*/2$ for any $j\in {\Bbb N}$ by
Proposition~\ref{pr5}. For the completeness of the proof, we may put
$R_0:=R_*.$~$\Box$
\end{proof}

\medskip
\begin{lemma}\label{lem1}
{\it Let $D$ be a domain in ${\Bbb R}^n,$ $n\geqslant 2,$ let $A$ be
a set in $D,$ and let $B(x_0, \varepsilon_1)\subset A$ for some
$\varepsilon_1>0.$

\medskip
Assume that, $Q\in L^1({\Bbb R}^n)$ and, in addition, for any
$y_0\in \overline{{\Bbb R}^n}$ there is
$\varepsilon_0=\varepsilon_0(y_0)>0$ and a Lebesgue measurable
function $\psi:(0, \varepsilon_0)\rightarrow [0,\infty]$ such that
\begin{equation}\label{eq7***} I(\varepsilon,
\varepsilon_0):=\int\limits_{\varepsilon}^{\varepsilon_0}\psi(t)\,dt
< \infty\quad \forall\,\,\varepsilon\in (0, \varepsilon_0)\,,\quad
I(\varepsilon, \varepsilon_0)\rightarrow
\infty\quad\text{при}\quad\varepsilon\rightarrow 0\,,
\end{equation}
and, in addition,
\begin{equation} \label{eq3.7.2}
\int\limits_{A(y_0, \varepsilon, \varepsilon_0)}
Q(y)\cdot\psi^{\,n}(|y-y_0|)\,dm(x) = o(I^n(\varepsilon,
\varepsilon_0))\,,\end{equation}
as $\varepsilon\rightarrow 0,$ where $A(y_0, \varepsilon,
\varepsilon_0)$ is defined in (\ref{eq1**}).
Then there is $r_0>0,$ which does not depend on $f,$ such that
\begin{equation}\label{eq1}
f(B(x_0, \varepsilon_1))\supset B(f(x_0), r_0)\qquad \forall\,\,f\in
\frak{F}_{E, \delta}(D)\,. \end{equation}}
\end{lemma}
\begin{remark}
If $y_0=\infty,$ the relation~(\ref{eq3.7.2}) must be understood by
the using the inversion $\psi(y)=\frac{y}{|y|^2}$ at the origin. In
other words, instead of
$$\int\limits_{A(y_0, \varepsilon, \varepsilon_0)}
Q(y)\cdot\psi^{\,n}(|y-y_0|)\,dm(y) = o(I^n(\varepsilon,
\varepsilon_0))$$
we need to consider the condition
$$\int\limits_{A(0, \varepsilon, \varepsilon_0)}
Q\left(\frac{y}{|y|^2}\right)\cdot\psi^{\,n}(|y|)\,dm(y) =
o(I^n(\varepsilon, \varepsilon_0))\,.$$
\end{remark}
{\it Proof of Lemma~\ref{lem1}.} Let us prove the lemma by
contradiction. Assume that its conclusion is wrong, i.e., the
relation~(\ref{eq1}) does not hold for any $r_0>0.$ Then for any
$m\in{\Bbb N}$ there is $y_m\in {\Bbb R}^n$ and $f_m\in \frak{F}_{E,
\delta}(D)$ such that $|f_m(x_0)-y_m|<1/m$ and $y_m\not\in
f_m(B(x_0, \varepsilon_1)).$ Due to the compactness of
$\overline{{\Bbb R}^n}$ we may consider that $y_m\rightarrow y_0$ as
$m\rightarrow\infty,$ where $y_0\in \overline{{\Bbb R}^n}.$ Then
also $f_m(x_0)\rightarrow y_0$ as $m\rightarrow\infty.$ We may
consider that $y_0\ne\infty.$

Since by the assumption $h(f_m(E))\geqslant \delta$ for any
$m\in{\Bbb N}$ and $d(f_m(E))\geqslant h(f_m(E)),$ where $d(f_m(E))$
denotes the Euclidean diameter of $f_m(E),$ there is
$\varepsilon_2>0$ such that
\begin{equation}\label{eq2}
f_m(E)\setminus B(y_0, \varepsilon_2)\ne \varnothing,\qquad
m=1,2,\ldots \,.
\end{equation}
By~(\ref{eq2}), there is $w_m=f_m(z_m)\in \overline{{\Bbb
R}^n}\setminus B(y_0, \varepsilon_2),$ where $z_m\in E.$ Since $E$
is a continuum, $\overline{{\Bbb R}^n}$ is a compactum and the set
$\overline{{\Bbb R}^n}\setminus B(y_0, \varepsilon_2)$ is closed, we
may consider that $z_m\rightarrow z_0\in E$ as $m\rightarrow\infty$
and $w_m\rightarrow w_0\in \overline{{\Bbb R}^n}\setminus B(y_0,
\varepsilon_2).$ Obviously, $w_0\ne y_0.$

\medskip
By Proposition~\ref{pr1} the family $f_m$ is equicontinuous. Now,
for any $\varepsilon>0$ there is $\delta=\delta(z_0)>0$ such that
$h(f_m(z_0), f_m(z))<\varepsilon$ whenever $|z-z_0|\leqslant
\delta.$ Then, by the triangle inequality
\begin{equation}\label{eq3}
h(f_m(z), w_0)\leqslant h(f_m(z), f_m(z_0))+ h(f_m(z_0), f_m(z_m))+
h(f_m(z_m), w_0)<3\varepsilon
\end{equation}
for $|z-z_0|<\delta,$ some $M_1\in {\Bbb N}$ and all $m\geqslant
M_1.$ We may consider that latter holds for any $m=1,2,\ldots .$
Since $w_0\in \overline{{\Bbb R}^n}\setminus B(y_0, \varepsilon_2),$
we may choose $\varepsilon>0$ such that $\overline{B_h(w_0,
3\varepsilon)}\cap \overline{B(y_0, \varepsilon_2)}=\varnothing,$
where $B_h(w_0, \varepsilon)=\{w\in \overline{{\Bbb R}^n}: h(w,
w_0)<\varepsilon\}.$ Then~(\ref{eq3}) implies that
\begin{equation}\label{eq4}
f_m(E_1)\cap \overline{B(y_0, \varepsilon_2)}=\varnothing,\qquad
m=1,2,\ldots \,,
\end{equation}
where $E_1:=\overline{B(z_0, \delta)}.$

Join the points $y_m$ and $f_m(x_0)$ by a segment $\beta_m:[0,
1]\rightarrow \overline{B(f_m(x_0), 1/m)}$ such that
$\beta_m(0)=f_m(x_0)$ and $\beta_m(0)=y_m.$ Let $\alpha_m,$
$\alpha_m:[0, c_m)\rightarrow B(x_0, \varepsilon_1),$ be a maximal
$f_m$-lifting of $\beta_m$ in $B(x_0, \varepsilon_1)$ starting at
$x_0.$ The lifting $\alpha_m$ exists by Proposition~\ref{pr3}. By
the same Proposition either $\alpha_m(t)\rightarrow x_1\in B(x_0,
\varepsilon_1)$ as $t\rightarrow c_m-0$ (in this case, $c_m=1$ and
$f_m(x_1)=y_m$), or $\alpha_m(t)\rightarrow S(x_0, \varepsilon_1)$
as $t\rightarrow c_m.$ Observe that, the first situation is
excluded. Indeed, if $f_m(x_1)=y_m,$ then $y_m\in f_m(B(x_0,
\varepsilon_1)),$ that contradicts the choice of $y_m.$ Thus,
$\alpha_m(t)\rightarrow S(x_0, \varepsilon_1)$ as $t\rightarrow
c_m.$ Observe that, $\overline{|\alpha_m|}$ is a continuum in
$\overline{B(x_0, \varepsilon_1)}$ and
$h(\overline{|\alpha_m|})\geqslant h(0, S(x_0, \varepsilon_1)).$
Let us to apply Lemma~\ref{lem2} for $A:=E_1:=B(z_0, \delta),$
$C_m:=|\alpha_m|$ and $r=h(0, S(x_0, \varepsilon_1)).$ By this lemma
we may find $R_0>0$ such that
\begin{equation}\label{eq5}
M(\Gamma(\overline{|\alpha_m|}, E_1, D))\geqslant R_0\,,\qquad
m=1,2,\ldots\,.
\end{equation}

Let us show that the relation~(\ref{eq5}) contradicts the definition
of the mapping $f_m$ in~(\ref{eq2*A})--(\ref{eqA2}). Indeed, since
$f_m(x_0)\rightarrow y_0$ as $m\rightarrow\infty,$ for any $k\in
{\Bbb N}$ there is a number $m_k\in {\Bbb N}$ such that
\begin{equation}\label{eq6}
B(f_{m_k}(x_0), 1/k)\subset B(y_0, 2^{\,-k})\,.
\end{equation}
Since $|\beta_m|\in B(f_m(x_0), 1/m),$ by~(\ref{eq6}) we obtain that
\begin{equation}\label{eq7}
|\beta_{m_k}|\subset B(y_0, 2^{\,-k})\,,\qquad k=1,2,\ldots\,.
\end{equation}
Let $k_0\in {\Bbb N}$ be such that $2^{\,-k}<\varepsilon_2,$ where
$\varepsilon_2$ is a number from~(\ref{eq4}), and let
$\Gamma_k:=\Gamma(|\alpha_{m_k}|, E_1, D).$ In this case, we observe
that
\begin{equation}\label{eq3G}
f_{m_k}(\Gamma_{k})>\Gamma(S(y_0, \varepsilon_2), S(y_0, 2^{\,-k}),
A(y_0, 2^{\,-k}, \varepsilon_2))\,,
\end{equation}
see Figure~\ref{fig1} for the scheme of the proof.
\begin{figure}[h]
\centerline{\includegraphics[scale=0.45]{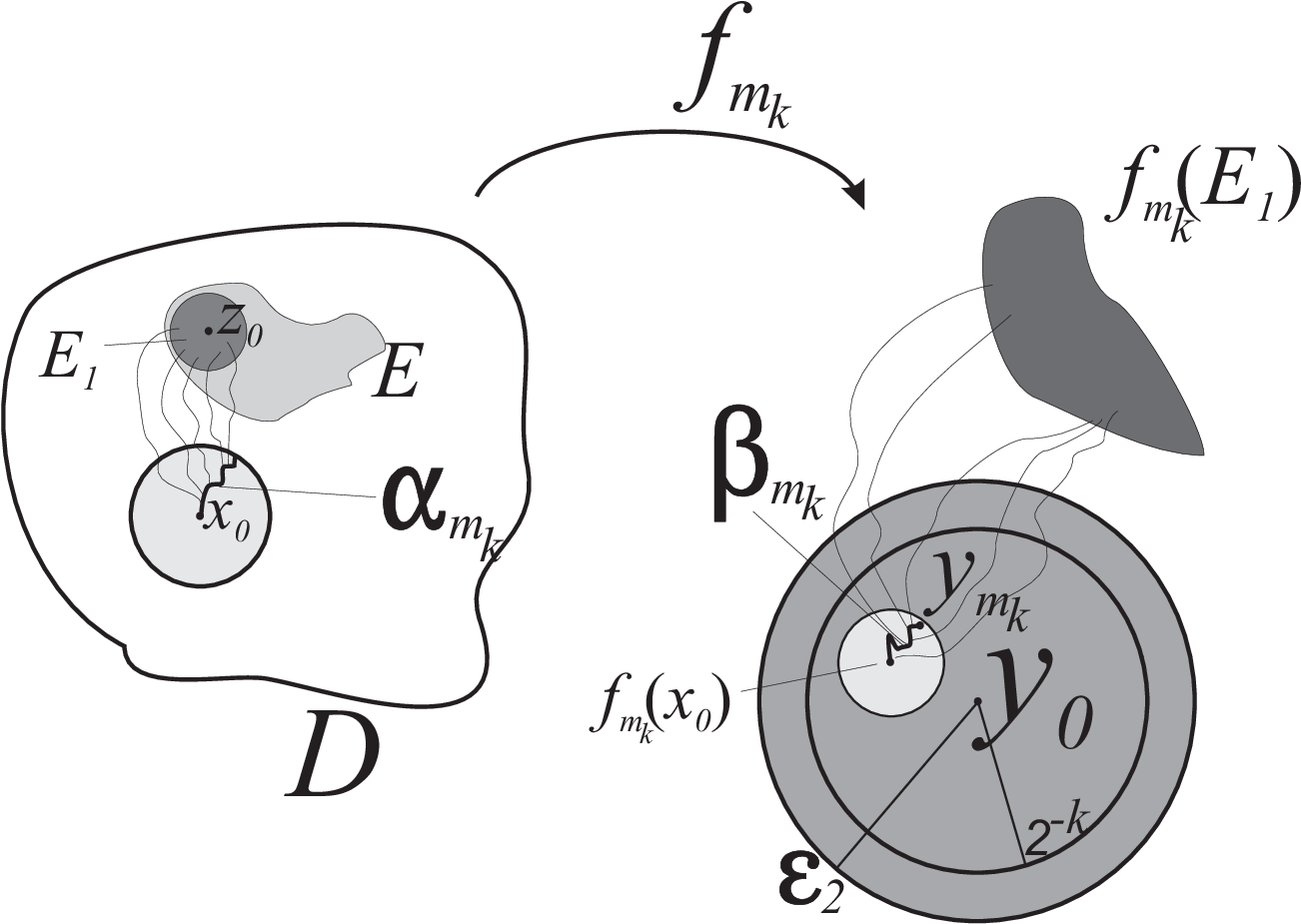}} \caption{To
the proof of Lemma~\ref{lem1}}\label{fig1}
\end{figure}
Indeed, let $\widetilde{\gamma}\in f_{m_k}(\Gamma_{k}).$ Then
$\widetilde{\gamma}(t)=f_{m_k}(\gamma(t)),$ where $\gamma\in
\Gamma_{k},$ $\gamma:[0, 1]\rightarrow D,$ $\gamma(0)\in
|\alpha_{m_k}|,$ $\gamma(1)\in E_1.$ By the relation~(\ref{eq4}), we
obtain that $f_{m_k}(\gamma(0))\in {\Bbb R}^n\setminus B(y_0,
\varepsilon_2).$ In addition, by~(\ref{eq7}) we have that
$f_{m_k}(\gamma(1))\in B(y_0, 2^{\,-k})\subset B(y_0,
\varepsilon_2)$ for $k\geqslant k_0.$ Thus,
$|f_{m_k}(\gamma(t))|\cap B(y_0, \varepsilon_2)\ne\varnothing \ne
|f_{m_k}(\gamma(t))|\cap ({\Bbb R}^n\setminus B(y_0,
\varepsilon_2)).$ Now, by~\cite[Theorem~1.I.5.46]{Ku} we obtain
that, there is $0<t_1<1$ such that $f_{m_k}(\gamma(t_1))\in S(y_0,
\varepsilon_2).$ Set $\gamma_1:=\gamma|_{[t_1, 1]}.$ We may consider
that $f_{m_k}(\gamma(t))\in B(y_0, \varepsilon_2)$ for any
$t\geqslant t_1.$ Arguing similarly, we obtain $t_2\in [t_1, 1]$
such that $f_{m_k}(\gamma(t_2))\in S(y_0, 2^{\,-k}).$ Put
$\gamma_2:=\gamma|_{[t_1, t_2]}.$ We may consider that
$f_{m_k}(\gamma(t))\not\in B(y_0, 2^{\,-k})$ for any $t\in [t_1,
t_2].$ Now, a path $f_{m_k}(\gamma_2)$ is a subpath of
$f(\gamma)=\widetilde{\gamma},$ which belongs to $\Gamma(S(y_0,
2^{\,-k}), S(y_0, \varepsilon_2), A(y_0, 2^{\,-k}, \varepsilon_2)).$
The relation~(\ref{eq3G}) is established.

\medskip
It follows from~(\ref{eq3G}) that
\begin{equation}\label{eq3H}
\Gamma_{k}>\Gamma_{f_{m_k}}(S(y_0, 2^{\,-k}), S(y_0, \varepsilon_2),
A(y_0, 2^{\,-k}, \varepsilon_2))\,.
\end{equation}
Since $I(\varepsilon, \varepsilon_0)\rightarrow\infty$ as
$\varepsilon\rightarrow 0,$ we may consider that $I(2^{\,-k},
\varepsilon_2)>0$ for sufficiently large $k\in {\Bbb N}.$ Set
$$\eta_{k}(t)=\left\{
\begin{array}{rr}
\psi(t)/I(2^{\,-k}, \varepsilon_2), & t\in (2^{\,-k}, \varepsilon_2)\,,\\
0,  &  t\not\in (2^{\,-k}, \varepsilon_2)\,,
\end{array}
\right. $$
where $I(2^{\,-k},
\varepsilon_2)=\int\limits_{2^{\,-k}}^{\varepsilon_2}\,\psi (t)\,
dt.$ Observe that
$\int\limits_{2^{\,-k}}^{\varepsilon_2}\eta_{k}(t)\,dt=1.$ Now, by
the relations~(\ref{eq3.7.2}) and~(\ref{eq3H}), and due to the
definition of $f_{m_k}$ in~(\ref{eq2*A})--(\ref{eqA2}), we obtain
that
$$M(\Gamma_{k})=M(\Gamma(|\alpha_{m_k}|, E_1, D)
)\leqslant M(\Gamma_{f_{m_k}}(S(y_0, 2^{\,-k}), S(y_0,
\varepsilon_2), A(y_0, 2^{\,-k}, \varepsilon_2)))\leqslant$$
\begin{equation}\label{eq3J}
\leqslant \frac{1}{I^n(2^{\,-k}, \varepsilon_2)}\int\limits_{A(y_0,
2^{\,-k}, \varepsilon_2)}
Q(y)\cdot\psi^{\,n}(|y-y_0|)\,dm(y)\rightarrow 0\quad \text{as}\quad
k\rightarrow\infty\,.
\end{equation}
The relation~(\ref{eq3J}) contradicts with~(\ref{eq5}). The
contradiction obtained above proves the lemma.~$\Box$

\medskip
{\it Proof of Theorem~\ref{th1}} immediately follows by
Lemma~\ref{lem1} and Proposition~\ref{pr6}.~$\Box$

\medskip
\begin{remark}\label{rem3}
If, under the conditions of Theorem~\ref{th1}, the mapped domain
$f(D)=D^{\,\prime}$ is fixed and bounded, then the condition $Q\in
L^1({\Bbb R^n})$ may be slightly weakened.

Given domain $D, D^{\,\prime}\subset {\Bbb R}^n,$ $n\geqslant 2,$ a
continuum $E\subset D,$ $\delta>0$ and a Lebesgue measurable
function $Q:{\Bbb R}^n\rightarrow [0, \infty]$ we denote by
$\frak{F}_{E, \delta}(D, D^{\,\prime})$ the family of all mapping
$f:D\rightarrow {\Bbb R}^n,$ $n\geqslant 2,$ satisfying
relations~(\ref{eq2*A})--(\ref{eqA2}) at any point $y_0\in
\overline{{\Bbb R}^n}$ such that $h(f(E))\geqslant \delta.$ The
following statement holds.

\medskip
{\it Assume that, $D^{\,\prime}$ is bounded and for each point
$y_0\in D^{\,\prime}$ and for every
$0<r_1<r_2<r_0:=\sup\limits_{y\in D^{\,\prime}}|y-y_0|$ there is a
set $E\subset[r_1, r_2]$ of a positive linear Lebesgue measure such
that the function $Q$ is integrable with respect to
$\mathcal{H}^{n-1}$ over the spheres $S(y_0, r)$ for every $r\in E.$
In addition, assume that one of the following conditions hold:

\medskip
1) $Q\in FMO(\overline{{\Bbb R}^n});$

\medskip
2) for any $y_0\in \overline{{\Bbb R}^n}$ there is $\delta(y_0)>0$
such that~(\ref{eq5D}) holds. Then there is $r_0>0,$ which does not
depend on $f,$ such that
$$f(B(x_0, \varepsilon_1))\supset B(f(x_0), r_0)\qquad \forall\,\,f\in \frak{F}_{E,
\delta}(D, D^{\,\prime})\,.$$}
The proof of this statement is exactly the same as the proof of
Theorem~\ref{th1}. The above condition on the function $Q,$
replacing the condition $Q\in L^1({\Bbb R^n}),$ ensures
equicontinuity of the family of mappings $\frak{F}_{E, \delta}(D,
D^{\,\prime})$ (see~\cite[Theorem~1.1]{SevSkv}). In all other
respects, the proof scheme is the same.
\end{remark}

\medskip
{\it Proof of Theorem~\ref{th2}.} Assume that $f$ is not a constant.
Now, the lightness of $f$ follows by~\cite[Theorem]{Sev$_1$}. It
remains to show that $f$ is open. Let $A$ be an open set and let
$x_0\in A.$ We need to show that, there is $\varepsilon^*>0$ such
that $B(f(x_0), \varepsilon^*)\subset f(A).$ Since $A$ is open,
there is $\varepsilon_1>0$ such that $\overline{B(x_0,
\varepsilon_1)}\subset A.$

Since $f$ is not constant, there are $a, b\in D$ such that $f(a)\ne
f(b).$ Let us join the points $a$ and $b$ by a path $\gamma$ in $D.$
We set $E:=|\gamma|.$ Now, $h(f_m(a), f_m(b))\geqslant
\frac{1}{2}\cdot h(f(a), f(b)):=\delta$ for sufficiently large $m\in
{\Bbb N}.$

By Theorem~\ref{th1} there is $r_0>0,$ which does not depend on $m,$
such that
$B(f_m(x_0), r_0)\subset f_m(B(x_0, \varepsilon_1)),$
$m=1,2,\ldots.$

Set $\varepsilon^*:=r_0/2.$ Let $y\in B(f(x_0), r_0/2).$ Since by
the assumption $f_m(x)\rightarrow f(x)$ locally uniformly in $D,$ by
the triangle inequality we obtain that
$$|f_m(x_0)-y|\leqslant |f_m(x_0)-f(x_0)|+|f(x_0)-y|<r_0$$
for sufficiently large $m\in {\Bbb N}.$ Thus, $y\in B(f_m(x_0),
r_0)\subset f_m(B(x_0, \varepsilon_1)).$ Consequently, $y=f_m(x_m)$
for some $x_m\in B(x_0, \varepsilon_1).$ Due to the compactness of
$\overline{B(x_0, \varepsilon_1)},$ we may consider that
$x_m\rightarrow z_0\in \overline{B(x_0, \varepsilon_1)}$ as
$m\rightarrow\infty.$ By the continuity of $f$ in $A,$ since
$\overline{B(x_0, \varepsilon_1)}\subset A,$ we obtain that
$f(x_m)\rightarrow f(z_0)$ as $m\rightarrow\infty.$ So, we have that
$f(x_m)\rightarrow f(z_0)$ as $m\rightarrow\infty$ and
simultaneously $y=f_m(x_m)$ for sufficiently large $m\in {\Bbb N}.$
Thus
$$|y-f(z_0)|=|f_m(x_m)-f(z_0)|\leqslant$$
$$\leqslant |f_m(x_m)-f(x_m)|+|f(x_m)-f(z_0)|\rightarrow 0\,,$$
$m\rightarrow\infty.$ Thus, $y=f(z_0)\in f(\overline{B(x_0,
\varepsilon_1)})\subset f(A).$ So, $y\in f(A),$ i.e., $B(f(x_0),
r_0/2)\subset f(A),$ as required.~$\Box$

\medskip
\begin{theorem}\label{th3}
{\it\, Let $D$ be a domain in ${\Bbb R}^n,$ $n\geqslant 2.$ Let
$f_j:D\rightarrow {\Bbb R}^n,$ $n\geqslant 2,$ $j=1,2,\ldots,$ be a
sequence of open discrete mappings satisfying the
conditions~(\ref{eq2*A})--(\ref{eqA2}) at any point $y_0\in
\overline{{\Bbb R}^n}$ and converging to some mapping
$f:D\rightarrow \overline{{\Bbb R}^n}$ as $j\rightarrow\infty$
locally uniformly in $D$ with respect to the chordal metric $h.$
Assume that the conditions on the function $Q$ from
Theorem~\ref{th1} hold. Then either $f$ is a constant in
$\overline{{\Bbb R}^n}$, or $f$ is light and open mapping
$f:D\rightarrow{\Bbb R}^n.$}
\end{theorem}

\medskip
\begin{proof}
Assume that $f$ is not a constant. Then there are $a, b\in D$ such
that $f(a)\ne f(b).$ Let us join the points $a$ and $b$ by a path
$\gamma$ in $D.$ We set $E:=|\gamma|.$ Now, $h(f_m(a),
f_m(b))\geqslant \frac{1}{2}\cdot h(f(a), f(b)):=\delta$ for
sufficiently large $m\in {\Bbb N}.$

Let $x_0\in D$ and let $y_0=f(x_0).$ By Theorem~\ref{th1} there is
$r_0>0,$ which does not depend on $m,$ such that
$B(f_m(x_0), r_0)\subset f_m(B(x_0, \varepsilon_1)),$
$m=1,2,\ldots.$ Then also $B_h(f_m(x_0), r_*)\subset f_m(B(x_0,
\varepsilon_1)),$ $m=1,2,\ldots,$ for some $r_*>0.$ Let $y\in
B_h(y_0, r_*/2)=B_h(f(x_0), r_*/2).$ By the converges of $f_m$ to
$f$ and by the triangle inequality, we obtain that
$$h(y, f_m(x_0))\leqslant h(y, f(x_0))+h(f(x_0), f_m(x_0))<r_*/2+r_*/2=r_*$$
for sufficiently large $m\in {\Bbb N}.$ Thus,
$$B_h(f(x_0), r_*/2)\subset B_h(f_m(x_0), r_*)\subset f_m(B(x_0,
\varepsilon_1))\subset {\Bbb R}^n\,.$$
In particular, $y_0=f(x_0)\in {\Bbb R}^n,$ as required. The
lightness and the openness of $f$ follows by
Theorem~\ref{th2}.~$\Box$
\end{proof}

\medskip
{\bf Open problem.} Is it possible to assert that, under the
conditions of Theorems~\ref{th2} and \ref{th3}, the mapping $f$ is
open and discrete?


CONTACT INFORMATION

\medskip
{\bf \noindent Evgeny Sevost'yanov} \\
{\bf 1.} Zhytomyr Ivan Franko State University,  \\
40 Bol'shaya Berdichevskaya Str., 10 008  Zhytomyr, UKRAINE \\
{\bf 2.} Institute of Applied Mathematics and Mechanics\\
of NAS of Ukraine, \\
19 Henerala Batyuka Str., 84 116 Slavyansk,  UKRAINE\\
esevostyanov2009@gmail.com

\medskip
{\bf \noindent Valery Targonskii} \\
Zhytomyr Ivan Franko State University,  \\
40 Bol'shaya Berdichevskaya Str., 10 008  Zhytomyr, UKRAINE \\
w.targonsk@gmail.com

\end{document}